\documentclass[12pt]{amsart}
\usepackage{amssymb,amsmath,amsthm,latexsym}
\usepackage{amssymb,graphicx}

\newtheorem{theorem}{Theorem}[section]

\newtheorem{lemma}[theorem]{Lemma}

\newtheorem{proposition}[theorem]{Proposition}

\theoremstyle{definition}
\newtheorem{definition}[theorem]{Definition}
\newtheorem{remark}[theorem]{Remark}
\newtheorem{question}[theorem]{Question}
\newtheorem{example}[theorem]{Example}
\theoremstyle{remark}

\title{On Property (FA) for Wreath Products}
\date{April 15, 2010}
\author{Yves Cornulier}
\address{IRMAR, Campus de Beaulieu, 35042 Rennes CEDEX, France}
\email{yves.decornulier@univ-rennes1.fr}
\author{Aditi Kar}
\address{School of Mathematics, University of Southampton, Southampton SO17 1BJ, UK}
\email{A.Kar@soton.ac.uk}
\thanks{The second author is supported by EPSRC grant  EP/F031947/1.}

\begin{document}

\begin{abstract} We characterize permutational wreath products with Property~(FA). For instance, the standard wreath product $A\wr B$ of two nontrivial countable groups $A,B$, has Property~(FA) if and only if $B$ has Property~(FA) and $A$ is a finitely generated group with finite abelianisation. We also prove an analogous result for hereditary Property~(FA). On the other hand, we prove that many wreath product with hereditary Property (FA) are not quotients of finitely presented groups with the same property.
\end{abstract}

\maketitle

\section{Introduction}

Property (FA) was introduced by Serre in his monograph \cite{S}: a group $G$ is said to have {\it Property~(FA)} if every isometric action of $G$ on a (simplicial) tree has a fixed point. Serre's fundamental result \cite[Theorem~I.6.15]{S} about Property (FA) says that a group $G$ has Property (FA) if and only if the three following conditions are satisfied
\begin{itemize}
\item $G$ is not a nontrivial amalgam;
\item $G$ has no quotient isomorphic to $\mathbf{Z}$;
\item $G$ is not the union of a properly increasing sequence of subgroups.
\end{itemize}
If $G$ is denumerable, the last condition is equivalent to the requirement that $G$ is finitely generated. In general, it is referred to in the literature as ``$G$ has cofinality $\neq\omega$" and is fulfilled by some uncountable groups \cite{KT}.
Traditional examples of finitely generated groups with Property (FA) include \begin{enumerate}\item finitely generated torsion groups; \item Coxeter groups defined by a Coxeter matrix with no occurrence of $\infty$; \item\label{sp} special linear groups over the integers, $\textnormal{SL}_n(\mathbf{Z})$, for $n \geq 3$; \item\label{ka} more generally, groups with Kazhdan's Property (T);\item\label{rers} irreducible lattices in semisimple Lie groups of real rank at least two, e.g.\ $\textnormal{SL}_2(\mathbf{Z}[\sqrt{2}])$.\end{enumerate}
The first three of these examples were explained by Serre in \cite{S}; (\ref{ka}) was proved by Watatani in \cite{W}, using the characterisation of property (T) in terms of affine actions on Hilbert spaces; and finally (\ref{rers}) was proved by Margulis (see \cite{margulis}). 

The aim of this article is to investigate Property (FA) for wreath products. We recall that given two groups $A$, $B$ and a $B$-set $X$, their {\it (permutational) wreath product} is defined as the group
$$A \wr_X B := A^{(X)}\rtimes B,$$ where $$A^{(X)}=\bigoplus_{x \in X} A_x$$ is the direct sum of isomorphic copies $A_x$ of~$A$ indexed by $X$. In the special case when $X=B$ with $B$ acting by left multiplication on itself, one obtains the {\it standard} wreath product and this, we denote simply as $A \wr B$. If $A$ and $B$ are finitely generated and $X$ has finitely many $B$-orbits, then $A\wr_X B$ is finitely generated as well. 

\begin{theorem}\label{t1} Consider the permutational wreath product $G=A \wr_X B$. Assume that $A\neq\{1\}$, $X\neq\emptyset$ and $X$ has finitely many $B$-orbits, each of which contains more than one element. The following are equivalent
\begin{itemize}
\item $G$ has Property (FA);
\item $B$ has Property (FA) and $A$ is a group with finite abelianisation, which cannot be expressed as the union of a properly increasing sequence of subgroups.
\end{itemize}\end{theorem}

Contrast with the following result on property (T) groups \cite[Proposition~2.8.2]{V}: the wreath product $A\wr B$ of two non-trivial groups $A,B$ has property (T) if and only if $A$ has property (T) and $B$ is finite.

The following is a well-known problem (it appears for instance as \cite[Question 7]{Cw} and in \cite{yves}).

\begin{question}[fg versus fp] \label{fp} Is every finitely generated group with Property (FA) the quotient of a finitely presented group with property~(FA)? \end{question}

It can also be stated as ``is Property (FA) open in the space of marked groups?" (see \cite[Section~2.6(h)]{CG}). The analogous question for some other fixed point properties has a positive answer 
\begin{itemize}
\item for Property (F$\mathbf{R}$) (fixed point property on $\mathbf{R}$-trees), a result of Culler and Morgan \cite[Proposition 4.1]{CM}.
\item Property (FH) (fixed point property on Hilbert spaces, also known as Kazhdan's Property (T)), a result independently due to Shalom and Gromov (\cite[Theorem 6.7]{Y} and \cite[3.8.B]{grom})
\item more generally, again by Gromov \cite[3.8.B]{grom}, the fixed point property on any class of metric spaces which is stable under ``scaling ultralimits", e.g.~the class of all $\textnormal{CAT}(0)$-spaces.
\end{itemize}

It is an old open question \cite[Question~A, p.286]{shalen} whether Property (FA) implies the {\it a priori} stronger Property (F$\mathbf{R}$). Of course a positive answer would imply a positive answer to Question \ref{fp}.

Some evidence for a positive answer for Question \ref{fp} is given by the case of wreath products. From the proof of Theorem \ref{t1}, one can deduce the proposition below. 
\begin{proposition}\label{fpfa}
Let $A$ and $B$ be finitely presented groups and let $X$ be a $B$-set with finitely many orbits. If in addition, $A$ has finite abelianisation and $B$ has Property (FA), then $A\wr_X B$ is the quotient of a finitely presented group with Property~(FA).
\end{proposition}
Note that Baumslag \cite{B} proved that a wreath product of non-trivial finitely presented groups $A\wr B$ is finitely presented only when $B$ is finite.

\begin{definition} A group $G$ has {\it hereditary Property (FA)} if $G$ and all its finite index subgroups have Property~(FA).
\end{definition}

It is natural to address Question \ref{fp} with Property (FA) replaced by hereditary (FA). In this situation, the answer turns out to be negative and wreath products provide a large class of elementary examples.

\begin{theorem}\label{lastt}
Let $G=A\wr B$ be the standard wreath product of two finitely generated groups. Assume that $B$ is an infinite, residually finite group and that $A$ has at least one non-trivial finite quotient. Then every finitely presented group mapping onto $G$ has a finite index subgroup with a surjective homomorphism onto a non-abelian free group.
\end{theorem}

The next theorem, which relies on Theorem \ref{t1} and further arguments, shows how to chose the group $G$ from Theorem \ref{lastt} to have hereditary (FA). 
\begin{theorem}\label{hfa}
Let $G=A\wr B$ be a wreath product of finitely generated groups, with $B$ infinite.
The following are equivalent
\begin{itemize}
\item $G$ has hereditary Property (FA);
\item $B$ has hereditary Property (FA) and $A$ has finite abelianisation.
\end{itemize}\end{theorem}

\begin{example}\label{exa}
If $G=F\wr \textnormal{SL}_3(\mathbf{Z})$ with $F$ any non-trivial finite group, then $G$ has hereditary Property (FA) by Theorem \ref{hfa}, but is not the quotient of any finitely presented group with the same property, by Theorem \ref{lastt}.
\end{example}

\begin{remark}Despite the analogy between Theorems \ref{t1} and \ref{hfa}, Theorem \ref{lastt} shows that Proposition \ref{fpfa} is false when (FA) is replaced by hereditary~(FA).\end{remark}

\begin{remark} Theorems \ref{lastt} and \ref{hfa} provide many instances (illustrated by Example \ref{exa}) of groups with hereditary Property~(FA), which are not quotients of finitely presented groups with hereditary Property~(FA). Here is another one, of a different kind. Let $\Gamma$ be the first Grigorchuk group \cite[Chap.~VIII]{slava}. This is a finitely generated group every proper quotient of which is finite; in particular it cannot be expressed as a non-trivial wreath product with an infinite quotient. Also, it is a finitely generated torsion group and therefore has hereditary Property~(FA). It follows however from \cite{grig} (see also \cite[Corollary 8]{bast}) that every finitely presented group mapping onto $\Gamma$ has a finite index subgroup mapping onto the free group. \end{remark}

\begin{remark}\label{vart}
Theorem \ref{t1} holds when Property (FA) replaced by (F$\mathbf{R}$), with a similar proof. 
\end{remark}

\begin{remark}\label{vartt}
It is not hard to extend Theorem \ref{hfa} to permutational wreath products. On the other hand, the extension of Theorem \ref{lastt} to permutational wreath products is more delicate.
\end{remark}

\noindent \textbf{Acknowledgements.}  We wish to thank Nikolay Nikolov, Armando Martino, Ashot Minasyan and Jo\"el Riou for valuable discussions and suggestions; we are grateful to Indira Chatterji, Luc Guyot and Alain Valette for reading the manuscript and their comments.

\section{Property (FA)} 

In this part, we prove Theorem \ref{t1} and Proposition \ref{fpfa}. If a group $G$ acts on a set $X$, we denote by $X^G$ the set of $G$-fixed points in $X$. We think of each tree as the set of its own vertices. We have the two following standard lemmas.

\begin{lemma} \label{lem} Suppose that a group $H$ acts on a tree $T$ without inversions. Let $A$ and $B$ be subgroups of $H$ such that $T^A$ and $T^B$ are non-empty. If $[A,B]=1$ then $T^A \cap T^B \neq \emptyset$. \end{lemma}
\begin{proof}On the contrary, suppose that $T^A \cap T^B$ is empty. 
Then, there is a unique geodesic segment $\alpha$ in $T$, of minimal length, joining $T^A$ and $T^B$. However, as $A$ and $B$ commute, $A$ preserves the set $T^B$. This implies that $A$ pointwise fixes the geodesic segment $\alpha$, hence fixes at least one element of $T^B$, which contradicts the assumption that $T^A \cap T^B$ is empty.\end{proof} 

We say that an action of a group $G$ on a tree $T$ is {\it parabolic} if every element of $G$ has a fixed point, but there is no global fixed point.

\begin{lemma}
If $G$ has a parabolic action on a tree, then $G$ is the union of a properly increasing sequence of subgroups.
\label{cof}
\end{lemma}
\begin{proof}
By a classical theorem of Tits \cite[Proposition~3.4]{tits}, there exists an end of $T$ which is strongly fixed by $G$. In other words, there exists a geodesic ray $(v_n)$ such that for every $g\in G$, $g.v_n=v_n$ for $n$ large enough. Define the non-decreasing sequence of stabilizers $$G_n=\{g\in G:\;g.v_k=v_k,\forall k\ge n\}.$$
By assumption, $G=\bigcup G_n$. But $G\neq G_n$ because $G$ has no global fixed point.
\end{proof}

\noindent The main idea of the proof of Theorem \ref{t1} is given by the following result.
\begin{proposition}\label{propess}
Let $A$ and $B$ be groups. Let $X$ be a $B$-set with $X^B=\emptyset$, and write $W=A^{(X)}$. Suppose that the permutational wreath product $G=A\wr_X B=W\rtimes B$ acts without inversions on a tree $T$. Assume that $T^B\neq\emptyset$, $T^W=\emptyset$, and that for any $x$, the action of $A_x$ on $T$ is not parabolic. Then there exists a unique geodesic line $\mathcal{L}\subset T$ preserved by $W$. Moreover, the $W$-action on $\mathcal{L}$ is non-trivial and by translations.
\end{proposition}
\begin{proof}
Let us first prove the proposition with the extra-assumption that $G$ acts transitively on $X$, and we fix a basepoint $o$ in $X$. 

\begin{itemize}
\item First, we prove that $T^{A_o}=\emptyset$. Assume the contrary. For each $b\in B$, we have $b.T^{A_o}=T^{A_{b.o}}$; in particular, by transitivity, the subtree $T^{A_x}$ is nonempty for each $x \in X$. Moreover, for $x \neq o$, $A_o$ and $A_x$ commute. Therefore, by Lemma \ref{lem}, we obtain $T^{A_o} \cap T^{A_x} \neq \emptyset$. 

By transitivity of the $G$-action on $X$, we see that $G$ is generated by $A_x$ and $B$ for any~$x\in X$. Accordingly, if for some $x$, $T^{A_x}\cap T^B\neq\emptyset$, then we deduce that $T^G\neq\emptyset$ and in particular $T^W\neq\emptyset$, a contradiction. 

Therefore, $T^{A_x}\cap T^B=\emptyset$ for any~$x\in X$. Denote by $u^x$ the unique vertex in $T^{A_x}$ closest to $T^B$. Note that we have $u^{b.x}=b.u^x$ for all $b\in B$.
For any vertex $v$ in $T$, denote by $(v_0,v_1,\dots,v_{\ell(v)}=v)$ the geodesic segment joining $T^B$ to $v$ (so $\ell(v)=d(v,T^B)$). Since $T$ is a tree, if for some $x\in X$, $v$ belongs to the subtree $T^{A_x}$ and $k_x =d(T^B,T^{A_x})$, we have $u^x=v_{k_x}$. Now \begin{align*}k_{b.x} &=d(T^B,T^{A_{b.x}})=d(b.T^B,b.T^{A_x})\\ &=d(T^B,T^{A_x})=k_x.\end{align*} 
Picking $v$ in $T^{A_o}\cap T^{A_{b.o}}$, we deduce that $u^{b.o}=v_{k_{b.o}}=v_{k_o}=u^o$; thus $u^{b.o}=u^o$ for all $b$. So $u^o$ is invariant under $B$, hence $T^{A_o}\cap T^B\neq\emptyset$, a contradiction.

\item Now we know that $T^{A_o} = \emptyset$. If every element of $A_o$ fixes some element of $T$, the action of $A_o$ on $T$ is, by definition, parabolic; this is ruled out by hypothesis. Accordingly, $A_o$ contains an element $a$ acting hyperbolically on $T$. Pick $b\in B$ with $b.o\neq o$; then $a'=bab^{-1}$ also acts hyperbolically on $T$. Let $\mathcal{L},\mathcal{L}'$ denote the axes of $a$ and $a'$. For any $x\in X-\{o\}$ (resp.\ $\in X-\{b.o\}$), $A_x$ centralises $a$ (resp.\ $a'$), so $A_x$ preserves $\mathcal{L}$ (resp.\ $\mathcal{L}'$). In particular, $a'$ preserves $\mathcal{L}$, but since $\mathcal{L}'$ is the unique axis preserved by $a'$, we deduce that $\mathcal{L}=\mathcal{L}'$. Therefore, for any $x\in X$, $A_x$ preserves $\mathcal{L}$. Now we claim that $A_o$ preserves the orientation of $\mathcal{L}$; otherwise, it contains some element $a''$ having a unique fixed point on $\mathcal{L}$. Since $a'$ centralizes $a''$, we deduce that this point is also fixed by $a'$, a contradiction since $a'$ acts by non-trivial translation on $\mathcal{L}$. By conjugating, we deduce that for each $x$, $A_x$ preserves the orientation of $\mathcal{L}$, so the whole action of $W$ on $\mathcal{L}$ is by translations and non-trivial.
\end{itemize}
Finally we have to tackle the non-transitive case. Denote by $X_i$ the (finitely many) $B$-orbits of $X$. Consider an action of $G$ as in the statement of the proposition. Then, for some $i$, $T^{A^{(X_i)}}\neq\emptyset$. By the transitive case, $A^{(X_i)}$ preserves a unique line $\mathcal{L}\subset T$, on which it acts non-trivially and by translations. Set $Y=X-X_i$. Then $A^{(Y)}$ centralizes $A^{(X_i)}$, so preserves $\mathcal{L}$. Moreover, $A^{(Y)}$ also preserves the orientation of $\mathcal{L}$ (otherwise as in the transitive case, it contains a point with a unique fixed point, which is then fixed by a hyperbolic element in $A^{(X_i)}$, a contradiction).
\end{proof}

\begin{proof}[Proof of Theorem \ref{t1}] Under the conditions of the theorem, suppose that the permutational wreath product $G:= A \wr_X B$ has Property (FA). Then clearly $B$, being a quotient of $G$, has Property (FA). Moreover, $A$ cannot be written as a properly increasing union of a sequence of subgroups $(A_n)$, since otherwise (using $X\neq\emptyset$) $G$ would be the increasing union of its subgroups $A_n\wr_X B$ and would fail to have Property~(FA). Moreover, since $X$ has no one-element orbit, the abelianisation of $G$ is given by $G^{\textnormal{ab}} = (A^{\textnormal{ab}})^Y \times B^{\textnormal{ab}}$, where $Y$ denotes the orbit set $B\backslash X$. Hence the abelianisation of $A$ is also finite. 

Conversely, suppose that $B$ has property (FA) and the group $A$, which has finite abelianisation, cannot be written as a properly increasing union of its subgroups. Let $G$ act without inversions on a tree $T$. To verify that $G$ has property (FA) we need to prove that $T^G$ is non-empty. 

Write $W=A^{(X)}$. Suppose that $T^W\neq\emptyset$. Then this is a nonempty subtree the action on which factors through $B$. So by Property~(FA) for $B$, we obtain $T^G\neq\emptyset$.

Otherwise, $T^W=\emptyset$; $T^B\neq\emptyset$ by Property~(FA) for $G$, and the action of $A_x$ is not parabolic by Lemma \ref{cof}. So Proposition \ref{propess} implies that there is a non-trivial homomorphism from $A^{(X)}$ to $\mathbf{Z}$. This is impossible since $\textnormal{Hom}(A^{(X)},\mathbf{Z})=\textnormal{Hom}(A,\mathbf{Z})^X=\{0\}$.
\end{proof}

\begin{lemma}[I.6.5.10 in \cite{S}] \label{tree} Let $T_1, \dots, T_m$ ($m\ge 2$) be subtrees of a tree $T$. If the $T_i$ have pairwise nonempty intersection, then their intersection is non-empty. \end{lemma}

\begin{proof}[Proof of Proposition \ref{fpfa}] We restrict ourselves to the case when $X$ is $B$-transitive; the reader can easily deduce the general case.

Fix a basepoint $o$ in $X$ and let $C\subset B$ be its stabilizer.
Fix finite generating subsets $S_A,S_B$ of $A$ and $B$. Consider the group $K$ obtained from the free product $A\ast B$ by adding the relators 
\begin{itemize}
\item[(1)] $[c,a]$ ($c\in C\cap S_B$, $a\in S_A$);
\item[(2)] $[bab^{-1},a']$ ($b\in S_B-C$, $a,a'\in S_A$). 
\end{itemize}
In $K$, for $b\in B$ write $A_b=bAb^{-1}$. Let $K$ act on a tree $T$. By Lemma \ref{tree}, it is enough to check that any pair of generators of $K$ has a common fixed point. Since $B$ has Property~(FA), every pair in $S_B$ has a common fixed point. The proof now proceeds along the same lines as that of Proposition~\ref{propess}.
\begin{itemize}
\item Suppose $T^{A_1}\neq\emptyset$. Thus, every pair of generators in $S_A$ has a common fixed point. It remains to show that any pair $(a,s)\in S_A\times S_B$ has a common fixed point.
Let us first check that $T^{A_1}\cap T^{A_s}\neq\emptyset$.
\begin{itemize}
\item If $s\in C$, then, thanks to relators of type~(1), we have $[s,A]=\{1\}$, hence $A_s=A_1$, so this is clear;
\item If $s\notin C$, then the relators of type~(2) assure us that $[A_{1},A_s]=\{1\}$. Hence by Lemma \ref{lem}, $A_1$ and $A_s$ have a common fixed point.
\end{itemize}
As in the proof of Theorem \ref{t1}, we denote the point in $T^{A_b}$ closest to $T^B$ as  $u^b$. Now, $u^1=s.u^1=u^s$ for all $s\in S_B$, so this is a fixed point for $B$. Thus, $a$ and $s$ have a common fixed point.
\item Suppose $T^{A_1}=\emptyset$. Observe that $S_B-C$ is non-empty else $X$ would have contained only one element. Pick $b\in S_B-C$. The subgroups $A_1$ and $A_b$ contain elements $a$, $a'$ respectively whose actions on $T$ are hyperbolic. The axis $\mathcal{L}$ of $a'$ is stabilised by $A_1$. Moreover every element of $A_1$ preserves the orientation of $\mathcal{L}$ for otherwise, its unique fixed point is a fixed point for $a'$. Thus the action of $A_1$ on $\mathcal{L}$ is by translations. Since $\textnormal{Hom}(A,\mathbf{Z})=\{0\}$, we deduce that $A_1$ acts trivially, contradicting the existence of $a$.\qedhere
\end{itemize}
\end{proof} 

\section{Hereditary Property (FA)}

\noindent We need the following classical lemma.
\begin{lemma}
Let $G_0$ be a group and let $(N_k)_k$ be a non-decreasing sequence of normal subgroups of $G_0$. Set $G=G_0/N$, where $N=\bigcup N_k$ and $G_k = G_0/N_k$. Let $H$ be a finitely presented group. Then every homomorphism $f:H\to G$ lifts to a homomorphism $f_k:H\to G_k$ for some $k$. If in addition, $f$ is surjective and $G_0$ is finitely generated, then $f_k$ can be chosen to be surjective. 
\end{lemma}
\begin{proof}
It is a standard result that a group $H$ is finitely presented (if and) only if the functor $\textnormal{Hom}(H,-)$ commutes with inductive limits (see \cite{AR}, Example~1.2(5) and Corollary~3.13, which reach far beyond the realm of groups). Since $G=\underrightarrow{\lim}\,G_k$, this proves the existence of an $f_k$. 

Suppose now that $f$ is surjective and $G_0$ has a finite generating subset~$S$. For every $s\in S$, the image of $s$ into $G$ belongs to the image of $f$; thus there exists $g_s$ in the kernel of $G_k\to G$ such that $sg_s$ belongs to $f_k(H)$. Since $S$ is finite, there exists $\ell\ge k$ such that $g_s=1$ in $G_\ell$ for all~$s\in S$. Thus the composite map $f_\ell:H\to G_k\to G_\ell$ lifts $f$ and is surjective.
\end{proof}

\begin{proof}[Proof of Theorem \ref{lastt}]
If $F$ is a finite subset of $B-\{1\}$, define
$\Gamma(A,B,F)$ as the quotient of $A\ast B$ by the ``relators" $[A,uAu^{-1}]$ for $u\le B$. Let $(u_k)_{k\ge 1}$ be an enumeration of $B-\{1\}$ and define, for $k\le\infty$
$$G_k=\Gamma(A,B,\{u_i:i\le k\}).$$
Note that $G_\infty=A\wr B$. 

The group $\Gamma(A,B,F)$ has a natural semidirect product decomposition $M\rtimes B$, where $M=M(A,B,F)$ is a ``graph product" (see \cite[Section~2]{Cwp}). This means that $M$ is the free product of copies $A_b$ of $A$ indexed by $b\in B$ and subject to the relations $[A_b, A_{bs}]=1$ for all $s\in F$, and $u\in B$ shifts $A_b$ to $A_{ub}$.  

Let $H$ be a finitely presented group having $G$ as a quotient. Then $H$ has $G_k$ as a quotient for some $k$. So we only have to prove that $G_k$ has a finite index subgroup mapping onto a free group.

We prove the following general statement. Let $A$ and $B$ be groups such that $A$ has at least one non-trivial finite quotient $A_1$ and $B$ is residually finite. Suppose $F$ is a symmetric subset of $B-\{1\}$ such that there exists $c,d\in B$ satisfying $\{c,d,c^{-1}d\}\cap (F\cup\{1\})=\emptyset$. Then $\Gamma(A,B,F)$ has a finite index subgroup mapping onto a non-abelian free group.

First observe that $\Gamma(A,B,F)$ maps onto $\Gamma(A_1,B,F)$, so replacing $A$ by $A_1$ if necessary, we can assume that $A$ is finite and non-trivial. Let $N$ be a normal subgroup of finite index in $B$ such that $F\cup \{1,c,d\}$ is mapped injectively into $B'=B/N$. Since the image of $F$ in $B'$ (still written $F$) is nontrivial, the group $\Gamma(A,B',F)$ is well-defined; this is a quotient of $\Gamma(A,B,F)$. Using the graph product description given above, we write $\Gamma(A,B',F)=M\rtimes B'$, with $M=M(A,B',F)$. So $M$ is a finite index subgroup of $\Gamma(A,B',F)$. Taking the quotient of $M$ by the normal subgroup generated by all $A_b$ for $b\neq\{1,c,d\}$, we see that all relators become trivial and therefore we obtain the free product $A_1\ast A_c\ast A_d$. The latter group has a non-abelian free subgroup of finite index.
\end{proof}

\begin{proposition}\label{p2}Consider the short exact sequence of groups:  \[ 1 \rightarrow A \rightarrow G \rightarrow B \rightarrow 1\] \noindent Assume that $A$ does not contain any nonabelian free subgroup, and that $G$ is finitely generated (or more generally, is not the union of a properly increasing sequence of subgroups) and does not map onto the integers or the infinite dihedral group. Then, $G$ has Property (FA) if and only if $B$ has Property (FA).
\end{proposition}

\begin{proof}
If $G$ is a group, define $\textnormal{NF}(G)$ to be the largest normal subgroup of $G$ without nonabelian free subgroups. The subgroup $\textnormal{NF}(G)$ is always well-defined. Indeed, let $N_1$ and $N_2$ be normal subgroups of $G$ with no nonabelian free subgroups. Then $N_1 N_2$ is also normal and the second isomorphism theorem implies that $N_1N_2$ cannot contain a non-abelian free subgroup.

The ``only if" part of the proposition is clear. Conversely, suppose that $G$ fails to have Property (FA). Then $G$ splits as a non-trivial amalgam $H\ast_K L$. If the amalgam were degenerate ($K$ has index two in both $H$ and $L$), then $G$ would map onto the infinite dihedral group. Therefore, we can apply \cite[Proposition~7]{Cicc}, which says in particular that $\textnormal{NF}(G)$ is contained in $K$. Since $A$ is by definition contained in $\textnormal{NF}(G)$, this shows that $G/A  \cong B$ splits as a non-trivial amalgam $(H/A)\ast_{K/A}L/A$, and therefore fails to have Property~(FA).
\end{proof}

\begin{proof}[Proof of Theorem \ref{hfa}]
The fact that the first condition implies the second one is as straightforward as the analogous implication for Theorem~\ref{t1}, so we do not repeat the argument.

So assume that $A$ has finite abelianisation and $B$ has hereditary Property~(FA).

We first prove the implication when $A$ has {\it trivial} abelianisation, as the proof is then easier. In this case, by Gruenberg \cite{Gruenberg} every finite index subgroup of $G$ contains the normal subgroup $A^{(B)}$ and is therefore of the form $A^{(B)}\rtimes C$ where $C$ has finite index in $B$; since $B$ is supposed to be infinite, $C$ is non-trivial. This group $A^{(B)}\rtimes C$ is a permutational wreath product (with a non-transitive free action), so Theorem \ref{t1} applies.

Before passing to the general case, we need to consider the special case when $A$ is abelian (and thus finite). Every finite index subgroup $H$ of $G$ then is an extension of groups such that the kernel $K$ is torsion (and abelian) and the quotient is a finite index subgroup of $B$. We claim that $H$ has no quotient $Q$ isomorphic to the group of integers or the infinite dihedral group. Suppose on the contrary that $Q \cong \mathbb{Z}$ or $D_\infty$. Since $Q$ has no non-trivial torsion normal subgroup, the image of $K$ into $Q$ is trivial, so $Q$ is a quotient of $B$. But this is absurd since $B$ has Property~(FA) by hypothesis. So we can apply Proposition \ref{p2} to deduce that $H$ has Property~(FA).

Suppose now, in general, that the derived subgroup $D$ of $A$ has finite index in $A$, and that $H$ has finite index in $G$; let $H$ act on a tree $T$. Then Gruenberg \cite{Gruenberg} implies that $H$ contains $D^{(B)}$. We claim that $T'=T^{D^{(B)}}\neq\emptyset$. Suppose $D^{(B)}$ has no global fixed point. Set $C=B\cap H$. The group $G'=D^{(B)}\rtimes C$ is a permutational wreath product; $B$ being a free $C$-set. Since $B$ is infinite, $C$ is non-trivial, so Proposition \ref{propess} applies. Therefore, $D^{(B)}$ preserves a unique line, on which it acts by non-trivial translations. Since $D^{(B)}$ is normal, this line is preserved by $G'$. The action on this line is given by a homomorphism from $G'$ to the infinite dihedral group. Since $D_\infty$ is residually finite, Gruenberg's theorem imples that the homomorphism is trivial on $D^{(B)}$. This is impossible and so $T'=T^{D^{(B)}}\neq\emptyset$.

Finally, since $D^{(B)}$ is normal in $H$,  we know that $T'$ is $H$-invariant and that the action of $H$ on $T'$ factors through $H/D^{(B)}$, which is a subgroup of finite index of $(A/D)\wr B$. By the special case when the base group is abelian, we deduce that there is a fixed point. This proves that $H$ has Property~(FA). \end{proof}

\end{document}